\documentclass[twoside]{article}
\usepackage{amssymb}
\usepackage{amsthm}
\usepackage{array}
\usepackage{amsfonts} 
\usepackage{amsmath}
\usepackage{graphicx,subcaption}

\newtheorem{theorem}{Theorem}[section]
\newtheorem{lemma}[theorem]{Lemma}
\newtheorem{proposition}[theorem]{Proposition}
\newtheorem{definition}[theorem]{Definition}
\newtheorem{remark}[theorem]{Remark}

\title{An $\MakeLowercase{(\underline{s},\overline{s}},S)$ optimal maintenance policy for systems subject to shocks and progressive deterioration} 

\author{Mauricio Junca\thanks{Department of Mathematics, Universidad de los Andes, Bogota, Colombia. Email address: {\tt mj.junca20@uniandes.edu.co}}}

\begin{document}

\maketitle

\begin{abstract}
We define a model of a system that deteriorates as a result of (i) shocks, modeled as a compound Poisson process and (ii) deterministic, state dependent progressive rate, with variable and fixed maintenance cost. We define maintenance strategies based on an impulse control model where time and size of interventions are executed according the system state, which is obtained from permanent monitoring. We characterize the value function as the unique viscosity solution of the HJB equation and prove that an $(\underline{s},\overline{s},S)$ policy is optimal. We also provide numerical examples. Finally, a singular control problem is proposed when there is no fixed cost, which study and relation with the former problem is open for future discussion.
\end{abstract}

\pagestyle{myheadings}

\section{Introduction}

In engineered systems, performance is related to some physical properties such as capacity, stiffness, mechanical resistance, etc., that decrease over time. In the structural reliability literature the structural condition at a given time is defined in practice by a single appropriate structural performance indicator even though several degradation factors may affect the system and cause decay in performance, see \cite{San11}. This performance decaying is known as \emph{deterioration} and it will eventually bring the system to a total failure condition. The use of stochastic models to describe deterioration and its effect in time in the system's performance is therefore important for reliability estimation.

There are different types of degradation processes that can be classified into three categories: shock-based, progressive and a combination of both degradation mechanisms. The first one models the effect of sudden events like earthquakes in infrastructures by reducing lump amounts of the system's performance in the form of jumps or shocks at discrete points in time. In the latter the system's performance is continuously removed over time and can model to processes like deterioration of pavements. \cite{Klutke02} and \cite{San11} propose some models for a combination of these two type of deterioration.

Based on the specific degradation mechanisms and models, different intervention policies can be proposed in order to improve the
availability or to extend the life of the system. Usually, maintenance policies are preventive or corrective. Preventive maintenance is frequently carried out without knowing the actual state of the system at the time of the intervention. In this case the decision variables may be the time of intervention, \cite{Klutke02}. On the other hand, corrective maintenance focuses on interventions once failure has been identified \cite{Kadi}.

This paper presents a maintenance policy based on stochastic control continuing the works of \cite{Junca} and \cite{Junca20131}. We assume that the system is continuously monitored, so maintenance can be executed at any time, and we do not consider any strategy after total failure. This kind of strategy can be considered as a preventive policy. The total failure state is such that no maintenance is possible and build or replace for a new system is necessary. Hence, this can be seen as a renewal epoch of a bigger renewal process model that combines preventice and corrective strategies, see \cite{nakagawa}.

The main goal of this paper is to characterize the value function using viscosity solutions theory and to prove that a $(\underline{s},\overline{s},S)$ policy is optimal for this problem. $(s,S)$ policies has been well established for inventory control problems with different models. In \cite{BenBen,BenLS05} the demand model is a mixture of Compound Poisson, diffusion and deterministic demand, while in \cite{BenJohn,BenJ09} the is no diffusion and the deterministic part is replaced by a function of the stock level, which is very similar to the model presented in this work. We note that the approach to prove policy optimality merely uses the definition of the value function and the form of the Hamilton Jacobi Bellman equation (or QVI) that the function satisfies rather than the analytical properties of the solution of the equation as in the previous references.

We organize the paper as follows: In Section \ref{dynamics} we describe the system's performance dynamics, in Section \ref{impulse} we define the impulse control problem and define the value function. Some properties of the function are also discussed in this section. Section \ref{viscosity} characterizes the value function as the unique viscosity solution of the associated HJB equation. The optimality of a $(\underline{s},\overline{s},S)$ policy is proved in Section \ref{policy}. Numerical examples are shown in Section \ref{numer}. We conclude in Section \ref{future}, propose a singular control problem for the no-fixed cost case and point put some directions for future work.

\section{System dynamics}\label{dynamics}

Consider a system whose performance is defined by a stochastic process $R$. Furthermore, assume that the system is subject to progressive deterioration and shocks. The shocks occur according to a Poisson process, and every shock causes a random amount of damage $S$. Total failure occurs when the system performance falls bellow a pre-defined threshold $m$ and there is a maximum performance level $M$ that cannot be improved. Formally, let $(\Omega,\mathcal{F},\mathbb{P})$ be the probability space in which we define all the stochastic quantities. We define the process $R=\{R_t\}_{t\geq0}$
\begin{equation}
R_t=r-\int_0^tc(R_s)ds-\sum\limits_{i=1}^{N_t}S_i,
\end{equation}
where the initial reliability level is $R_{0-}=r$, $c$ is a positive and non-increasing Lipschitz function with constant $L$ on $[m,M]$, $N=\{N_t\}_{t\geq0}$ is an homogeneous Poisson process with intensity $\lambda>0$, and the sequence of the sizes of the shocks $\{S_i\}_{i\in\mathbb{N}}$ are independent and identically distributed random variables with probability distribution $F$ on $[0,\infty)$ and independent of the Poisson process $N$. Note that $c$ is the rate of deterioration that is assume to be higher for small values of $R_t$ than for large values, that is, the better the system the bigger its resistance. We also assume that the support of $F$ contains the interval $[0,M-m]$.

\section{Impulse control}\label{impulse}

The impulse control model for maintenance is studied for the first time in \cite{Junca}. This type of control is defined as follows:
\begin{definition}
A \textit{maintenance policy} for the system is a double sequence $\nu=\{(\tau_i,\zeta_i)\}_{i\in\mathbb{N}}$  of intervention times $\tau_i$ at which the performance is improved an amount $\zeta_i$. The policy is an \textit{impulse control} if satisfies the following conditions:
\begin{enumerate}
\item $0\leq\tau_i\leq\tau_{i+1}$ a.s. for all $i\in\mathbb{N}$,
\item $\tau_i$ is a stopping time with respect to the filtration $\mathcal{F}_t=\sigma\{R_{s}|s\leq t\}$ for $t\geq0$,
\item $\zeta_i$ is a $\mathcal{F}_{\tau_i}$-measurable random variable.
\item $\lim\limits_{i\rightarrow\infty}\tau_i=\infty$ a.s.
\end{enumerate}
\end{definition}

Given an impulse control $\nu$, the controlled process $R^{\nu}=\{R^{\nu}\}_{t\geq0}$ is defined by
\begin{equation}
R_t^{\nu}=r-\int_0^tc(R_s^{\nu})ds-\sum\limits_{i=1}^{N_t}S_i+\sum\limits_{\tau_i< t}\zeta_i.
\end{equation}

The time of total failure of the controlled process is
\begin{equation}
\tau^{\nu}=\inf\{t\geq0|R^{\nu}_t< m\},
\end{equation}
and it is assumed that if the system reaches the threshold the process is stopped. We denote by $\tau^0$ the time of total failure of the uncontrolled process $R$.

Any intervention at time $\tau_i$ depends on the state of the system at this time, that is $R_{\tau_i}$. Since the state of the system cannot be above $M$, the set of possible actions is $[0,M-R_{\tau_i}]$. We call $\mathcal{I}$ the set of admissible impulse controls such that $\zeta_i\in[0,M-R_{\tau_i}]$ for all $i\in\mathbb{N}$. Each intervention has a cost given by $C(r,\zeta)$, which is the cost of bringing the system from level $r$ to level $r+\zeta$. We assume that $C(r,\zeta)=H(r+\zeta)-H(r)+k$, with $H$ a continuous and increasing function and $k>0$ is the fixed cost.

On the other hand, there is a benefit for keeping the system at a given state $r$ denoted by $G(r)$, where $G$ is a non-negative continuous, increasing function on $[m,M]$.

If we denote $\mathbb{E}_{r}[\cdot]:=\mathbb{E}[\cdot |R^{\nu}_{0-}=r]$, for a given $\nu\in\mathcal{I}$ and initial component state $r\in[m,M]$, the expected discounted profit can be computed as
\begin{equation}
J(r,\nu)=\mathbb{E}_{r}\left[\int_0^{\tau^{\nu}}e^{-\delta s}G(R_s^{\nu})ds-\sum\limits_{\tau_i<\tau^{\nu}}e^{-\delta\tau_i}C(R^{\nu}_{\tau_i},\zeta_i)\right],
\label{CBeq}
\end{equation}
where $\delta>0$ is the discount factor. 

The value function is defined as
\begin{equation}\label{valimp}
V(r)=\sup_{\nu\in\mathcal{I}}J(r,\nu)
\end{equation}
 and it is discussed in detail in \cite{Junca,Junca20131} for the case when there is no progressive deterioration. We will show that $V$ is the unique bounded viscosity solution to the corresponding Hamilton-Jacobi-Bellman (HJB) equation

\begin{equation}\label{hjbimp}
\min\{\delta f(r)-\mathcal{A}f(r)-G(r),f(r)-\mathcal{M}f(r)\}=0,
\end{equation}
for all $r\in[m,M]$, where 
\begin{equation}\label{infinoper}
\mathcal{A}f(r)=-c(r)f'(r)+\lambda\left(\int_0^{r-m}f(r-s)dF(s)-f(r)\right)
\end{equation}
and
\begin{equation}\label{interoper}
\mathcal{M}f(r)=\sup_{0\leq\zeta\leq M-r}f(r+\zeta)-H(r+\zeta)+H(r)-k.
\end{equation}

From the definition of the value function $V$, we can see that it is non-negative and bounded. Also, if we define $V$ for $r<m$, we have that $V(r)=0$. Another important property is the continuity as we will see next.

\begin{lemma}\label{contlemma}
Let $R_t^r$ denote the uncontrolled process with initial level $r$. Let $\tau\leq T<\infty$ be a stopping time, then 
$$\lim\limits_{r'\rightarrow r}R_{\tau}^{r'}=R_{\tau}^r\textrm{ a.s.}$$
\end{lemma}
\begin{proof}
First note that before any shock occurs we have that
\begin{align*}
\frac{d}{dt}|R_t^{r'}-R_t^r|&\leq\left|\frac{dR_t^{r'}}{dt}-\frac{dR_t^{r}}{dt}\right|\\
&=|c(R_t^{r'})-c(R_t^{r})|\leq L|R_t^{r'}-R_t^{r'}|.
\end{align*}
By Gronwall's inequality we get that $|R_t^{r'}-R_t^r|\leq|r'-r|e^{Lt}$. Now, for each $\omega\in\Omega$, with shock times $T_1(\omega), T_2(\omega),\ldots$, there is $n$ such that $T_n(\omega)\leq \tau(\omega)<T_{n+1)}(\omega)$. Hence
\begin{align*}
|R_{\tau(\omega)}^{r'}-R_{\tau(\omega)}^r|&\leq|R_{T_n(\omega)}^{r'}-R_{T_n(\omega)}^r|e^{L(\tau(\omega)-T_n(\omega))}\\
&\leq|R_{T_{n-1}(\omega)}^{r'}-R_{T_{n-1}(\omega)}^r|e^{L(\tau(\omega)-T_{n-1}(\omega))}\\
&\leq|r'-r|e^{L\tau(\omega)}\leq|r-r'|e^{LT}.
\end{align*}
This proves the lemma.
\end{proof}

This lemma allows us to prove the following proposition:

\begin{proposition}\label{continuity}
The value function $V$ is non-decreasing and continuous on $[m,M]$.
\end{proposition}
\begin{proof}
We will first show that the value function is non-decreasing. Let $m\leq r<r'\leq M$ and $\nu=\{(\tau_i,\zeta_i)\}\in\mathcal{I}$ with initial level $r$. Define the policy $\bar{\nu}$ with initial level $r'$ as following: Consider $i^*=\min\{i: R^{\nu}_{\tau_i}+\zeta_i>R^{r'}_{\tau_i}\}$, where $R_t^{r'}$ denotes the uncontrolled process with initial level $r'$. Then, let $\bar{\nu}=\{(\bar{\tau}_i,\bar{\zeta}_i)\}$ with $\bar{\tau}_1=\tau_{i^*}$, $\bar{\zeta}_1=R^{\nu}_{\tau_i}+\zeta_i-R^{r'}_{\tau_i}$ and for $i>1$ $\bar{\tau}_i=\tau_{i^*+i-1}$ and $\bar{\zeta}_i=\zeta_{i^*+i-1}$ in the case $i^*<\infty$. If $i^*=\infty$ then do nothing. Hence, it is easy to see that $R_t^{\nu}\leq R_t^{\bar{\nu}}$ for all $t\geq0$ and 
also the maintenance cost of the policy $\nu$ is greater than the cost of policy $\bar{\nu}$. Therefore, $J(r,\nu)\leq J(r',\bar{\nu})\leq V(r')$. Since the policy $\nu$ is arbitrary we get that $V(r)\leq V(r')$.

Consider now an $\epsilon-$optimal policy $\nu=\{(\tau_i,\zeta_i)\}$ with initial level $r'$, that is, $V(r')\leq J(r',\nu)+\epsilon$. We define the policy $\bar{\nu}$, with initial level $r$, equal to $\nu$ except for the size of the first intervention: $\bar{\zeta}_1=\zeta_1+R^{r'}_{\tau_1}-R^{r}_{\tau_1}$.

Now, we consider two disjoint events:
\begin{itemize}
\item $\{\tau^{\bar{\nu}}<\tau^{\nu}\}$: In this case we must have that $\tau^{\bar{\nu}}=\tau^0$ and $R^{r}_{\tau^{\bar{\nu}}}<0\leq R^{r'}_{\tau^{\bar{\nu}}}$. Since $\tau^0$ is bounded due to progressive deterioration, by Lemma \ref{contlemma} 
\begin{equation}\label{th1}
\lim\limits_{r\rightarrow r'-}\mathbb{P}\left(\tau^{\bar{\nu}}<\tau^{\nu}\right)=0.
\end{equation}
\item $\{\tau^{\bar{\nu}}=\tau^{\nu}\}$: In this case we have two possibilities, $\tau^{\nu}\leq\tau_1$ or $\tau^{\nu}>\tau_1$. In the second case $R_s^{\bar{\nu}}=R_s^{\nu}$ for $s\geq\tau_1$.
\end{itemize}

From both cases above we have that
\begin{align}\label{th2}
V(r)&\geq J(r,\bar{\nu})=\mathbb{E}\left[1_{\{\tau^{\bar{\nu}}<\tau^{\nu}\}}\int_0^{\tau^{\bar{\nu}}}e^{-\delta s}G(R_s^{r})ds\right]\\\nonumber
&+\mathbb{E}\left[1_{\{\tau^{\bar{\nu}}=\tau^{\nu}\}}\int_0^{\tau^{\nu}\wedge\tau_1}e^{-\delta s}G(R_s^{r})ds\right]-\mathbb{E}\left[1_{\{\tau^{\nu}>\tau_1\}}e^{-\delta\tau_1}C(R^{r}_{\tau_1},\bar{\zeta}_1)\right]\\\nonumber
&+\mathbb{E}\left[1_{\{\tau^{\nu}>\tau_1\}}\left(\int_{\tau_1}^{\tau^{\nu}}e^{-\delta s}G(R_s^{\nu})ds-\sum\limits_{\tau_1<\tau_i<\tau^{\bar{\nu}}}e^{-\delta\tau_i}C(R^{\nu}_{\tau_i},\zeta_i)\right)\right].\\\nonumber
\end{align}
On the other hand
\begin{align}\label{th3}
V(r')-\epsilon&\leq J(r',\nu)=\mathbb{E}\left[1_{\{\tau^{\bar{\nu}}<\tau^{\nu}\}}\left(\int_0^{\tau^{\nu}}e^{-\delta s}G(R_s^{\nu})ds-\sum\limits_{\tau_i<\tau^{\bar{\nu}}}e^{-\delta\tau_i}C(R^{\nu}_{\tau_i},\zeta_i)\right)\right]\\\nonumber
&+\mathbb{E}\left[1_{\{\tau^{\bar{\nu}}=\tau^{\nu}\}}\int_0^{\tau^{\nu}\wedge\tau_1}e^{-\delta s}G(R_s^{r'})ds\right]-\mathbb{E}\left[1_{\{\tau^{\nu}>\tau_1\}}e^{-\delta\tau_1}C(R^{r'}_{\tau_1},\zeta_1)\right]\\\nonumber
&+\mathbb{E}\left[1_{\{\tau^{\nu}>\tau_1\}}\left(\int_{\tau_1}^{\tau^{\nu}}e^{-\delta s}G(R_s^{\nu})ds-\sum\limits_{\tau_1<\tau_i<\tau^{\bar{\nu}}}e^{-\delta\tau_i}C(R^{\nu}_{\tau_i},\zeta_i)\right)\right].\\\nonumber
\end{align}

Combining \eqref{th2} and \eqref{th3} we have that
\begin{align}\label{th4}
V(r)&\geq\mathbb{E}\left[1_{\{\tau^{\bar{\nu}}<\tau^{\nu}\}}\int_0^{\tau^{\bar{\nu}}}e^{-\delta s}G(R_s^{r})ds\right]\\\label{th5}
&+\mathbb{E}\left[1_{\{\tau^{\bar{\nu}}=\tau^{\nu}\}}\int_0^{\tau^{\nu}\wedge\tau_1}e^{-\delta s}\left(G(R_s^{r})-G(R_s^{r'})\right)ds\right]\\\label{th6}
&-\mathbb{E}\left[1_{\{\tau^{\bar{\nu}}<\tau^{\nu}\}}\left(\int_0^{\tau^{\nu}}e^{-\delta s}G(R_s^{\nu})ds-\sum\limits_{\tau_i<\tau^{\nu}}e^{-\delta\tau_i}C(R^{\nu}_{\tau_i},\zeta_i)\right)\right]\\\label{th7}
&+\mathbb{E}\left[1_{\{\tau^{\nu}>\tau_1\}}e^{-\delta\tau_1}\left(H(R^r_{\tau_1})-H(R^{r'}_{\tau_1})\right)\right]\\\nonumber
&+V(r')-\epsilon.
\end{align}

From \eqref{th1} the terms \eqref{th4} and \eqref{th6} converges to 0 as $r\rightarrow r'-$. Also, by Lemma \ref{contlemma}, bounded convergence theorem and continuity of the functions $G$ and $H$, the terms \eqref{th5} and \eqref{th7} converges to 0. Similarly, we get the same result when $r'\rightarrow r+$. Since $\epsilon$ is arbitrary, the proposition follows by the non-decreasing property.
\end{proof}

\begin{remark}\label{remcont} An important consequence of the theorem is that $\mathcal{M}V$ is also a continuous function, as can be easily checked.
\end{remark}

\section{Viscosity solution}\label{viscosity}

In this section we prove that the value function $V$, defined in \eqref{valimp}, is a solution of the HJB equation \eqref{hjbimp}. Since we do not know the exact regularity of the $V$, we need the appropriate notion of solution. We consider the notion of viscosity solution. Viscosity solutions were introduce in \cite{Crandall} for first-order Hamilton Jacobi equations. Later, different generalizations were proposed, for example in \cite{Azcue} it is defined for first-order integro-differential equations. We will use the following definition:

\begin{definition}\label{defviscosityimp}
\begin{enumerate}
\item[(i)] A \emph{viscosity subsolution} of \eqref{hjbimp} is an upper semi-continuous function on $[m,M]$ $u$ such that for each $\varphi\in C^1(m,M)$,
$$\min\left\{\delta \varphi(r)-\mathcal{A}\varphi(r)-G(r),u(r)-\mathcal{M}u(r)\right\}\leq0$$ 
at every $r\in(m,M)$ which is a maximizer of $u-\varphi$ with $u(r)=\varphi(r)$.

\item[(ii)] A \emph{viscosity supersolution} of \eqref{hjbimp} is an lower semi-continuous function on $[m,M]$ $v$ such that for each $\phi\in C^1(m,M)$,
$$\min\left\{\delta \phi(r)-\mathcal{A}\phi(r)-G(r),v(r)-\mathcal{M}v(r)\right\}\geq0$$ 
at every $r\in(m,M)$ which is a minimizer of $v-\phi$ with $v(r)=\phi(r)$.

\item[(iii)] A function is a \emph{viscosity solution} of \eqref{hjbimp} if it is both a viscosity subsolution and a viscosity supersolution.
\end{enumerate}
\end{definition}

Before proving that $V$ is a viscosity solution of \eqref{hjbimp} we need the following lemma proved in \cite{Junca}. We include the proof in \ref{applema}. 

\begin{lemma}\label{lema}
Let $\tau$ be a stopping time. Then for all $r\in[m,M]$
\begin{equation}\label{dynamic}
V(r)\geq\mathbb{E}_r\left[\int_0^{\tau\wedge\tau^0}e^{-\delta s}G(R^r_s)ds+1_{\{\tau<\tau^0\}}e^{-\delta\tau}V(R^r_{\tau})\right].
\end{equation}
Furthermore, we have equality in \eqref{dynamic} if it is not optimal to intervene the system before $\tau$.
\end{lemma}

\begin{theorem}\label{viscoimp}
The value function $V$ defined by \eqref{valimp} is a viscosity solution of equation \eqref{hjbimp}.
\end{theorem}
\begin{proof}
First of all, Theorem \ref{continuity} establishes the continuity of $V$, so it is both lower and upper semi-continuous. To prove the subsolution property let $r_0\in(m,M)$ and $\varphi\in C^1$ such that $V-\varphi\leq0=V(r_0)-\varphi(r_0)$. Suppose by contradiction that $V$ is not a viscosity subsolution of \eqref{hjbimp}, hence, by Remark \ref{remcont}, there exist $\alpha>0$ and $\epsilon>0$ such that 
\begin{equation}\label{sub1}
\delta \varphi(r)-\mathcal{A}\varphi(r)-G(r)>\alpha
\end{equation}
and
\begin{equation}\label{sub2}
V(r)-\mathcal{M}V(r)>\alpha,
\end{equation}
for all $r\in(r_0-\epsilon,r_0+\epsilon)\subset(m,M)$. \eqref{sub2} implies that it is not optimal to intervene the system for levels in such interval. Let $\bar{\tau}=\inf\{s\geq0:|R^{r_0}_s-r_0|\geq\epsilon\}$. By Dynkin's formula (see \cite{Junca}) we have that
\begin{align*}
\mathbb{E}_{r_0}\left[e^{-\delta\bar{\tau}}V(R_{\bar{\tau}})\right]&\leq\mathbb{E}_{r_0}\left[e^{-\delta\bar{\tau}}\varphi(R_{\bar{\tau}})\right]\\
&=\varphi(r_0)+\mathbb{E}_{r_0}\left[\int_0^{\bar{\tau}}e^{-\delta s}\left(\mathcal{A}\varphi(R_s)-\delta\varphi(R_s)\right)ds\right]\\
&\leq V(r_0)-\mathbb{E}_{r_0}\left[\int_0^{\bar{\tau}}e^{-\delta s}\left(\alpha+G(R_s)\right)ds\right]\\
&=\mathbb{E}_{r_0}\left[e^{-\delta\bar{\tau}}V(R_{\bar{\tau}})\right]-\alpha\mathbb{E}_{r_0}\left[\int_0^{\bar{\tau}}e^{-\delta s}ds\right],
\end{align*}
where the last equality follows from Lemma \ref{lema}. Since $\alpha>0$, this implies $\bar{\tau}=0$ a.s. which is impossible.

To establish the supersolution property let  $r_0\in(m,M)$ and $\phi\in C^1$ such that $V-\phi\geq0=V(r_0)-\phi(r_0)$. It is always true that $V-\mathcal{M}V\geq0$, so suppose by contradiction that there exist $\alpha>0$ and $\epsilon>0$ such that 
\begin{equation}\label{sup1}
\delta \phi(r)-\mathcal{A}\phi(r)-G(r)<-\alpha
\end{equation}
for all $r\in(r_0-\epsilon,r_0+\epsilon)\subset(m,M)$.  Let $\bar{\tau}$ be defined as before. Hence
\begin{align*}
\mathbb{E}_{r_0}\left[e^{-\delta\bar{\tau}}\phi(R_{\bar{\tau}})\right]&=\phi(r_0)+\mathbb{E}_{r_0}\left[\int_0^{\bar{\tau}}e^{-\delta s}\left(\mathcal{A}\phi(R_s)-\delta\phi(R_s)\right)ds\right]\\
&\geq V(r_0)+\mathbb{E}_{r_0}\left[\int_0^{\bar{\tau}}e^{-\delta s}\left(\alpha-G(R_s)\right)ds\right]\\
&\geq\mathbb{E}_{r_0}\left[e^{-\delta\bar{\tau}}V(R_{\bar{\tau}})\right]+\alpha\mathbb{E}_{r_0}\left[\int_0^{\bar{\tau}}e^{-\delta s}ds\right]\\
&\geq\mathbb{E}_{r_0}\left[e^{-\delta\bar{\tau}}\phi(R_{\bar{\tau}})\right]+\alpha\mathbb{E}_{r_0}\left[\int_0^{\bar{\tau}}e^{-\delta s}ds\right],
\end{align*}
where the second inequality follows from Lemma \ref{lema}. We have the same contradiction as before and therefore the theorem is proved.
\end{proof}

\begin{remark}\label{k0} The proof above woks the same even if there is no fixed cost, that is if $k=0$.\end{remark}

\subsection{Uniqueness}

Now, we are going to characterize the value function $V$ among all viscosity solution of \eqref{hjbimp}. The first step is to prove a \emph{comparison principle} for this equation. In order to do this we define the operator
\begin{equation}\label{infinoper2}
\mathcal{A}(f,u)(r)=-c(r)f'(r)+\lambda\left(\int_0^{r-m}u(r-s)dF(s)-u(r)\right),
\end{equation}
for $f\in C^1$ and $u$ continuous. An equivalent definition of viscosity solution can be given in terms of this operator instead of $\mathcal{A}$ (see \cite{Azcue}). Using some ideas of \cite{Seydel} we have the following comparison principle:
\begin{proposition}\label{uniqimp}
Let $u$ be a viscosity subsolution of \eqref{hjbimp} and $v$ be a viscosity supersolution of \eqref{hjbimp} such that $u(m)\leq v(m)$ and $u(M)\leq v(M)$. Then $u\leq v$ in $[m,M]$.
\end{proposition}
\begin{proof}
Let $\gamma\geq \dfrac{k+\max\limits_{m\leq r\leq M}G(r)}{\delta}$ and for $n\geq 1$ define $v_n(r)=\left(1-\dfrac{1}{n}\right)v(r)+\dfrac{\gamma}{n}$. Now, suppose $\phi_n-v_n\leq0=\phi_n(r_0)-v_n(r_0)$. Let $\phi=\left(\phi_n-\frac{\gamma}{n}\right)\frac{n}{n-1}$, so $\phi-v\leq0=\phi(r_0)-v(r_0)$. Therefore, the supersolution property of $v$ implies that
\begin{equation}\label{superk}
\min\left\{\delta \phi_n(r_0)-\mathcal{A}(\phi_n,v_n)(r_0)-G(r_0),v_n(r_0)-\mathcal{M}v_n(r_0)\right\}\geq\frac{k}{n}.
\end{equation}
Now, if $u\leq v_n$ for all $n$, then, taking $n\rightarrow\infty$ we get that $u\leq v$. Hence, suppose by contradiction that for some $n$, $\max\limits_{m\leq r\leq M} u(r)-v_n(r)=\kappa>0.$ Following the usual idea of doubling the number of variables as in \cite{crisli}, given $\alpha>0$ let
\begin{equation}\label{ka}
\kappa_{\alpha}=\max\limits_{[m,M]\times[m,M]}u(r)-v_n(s)-\frac{\alpha}{2}(r-s)^2=u(r_{\alpha})-v_n(s_{\alpha})-\frac{\alpha}{2}(r_{\alpha}-s_{\alpha})^2.
\end{equation}
Then, by Lemma 3.1 in \cite{crisli}, $\alpha(r_{\alpha}-s_{\alpha})^2\rightarrow0$ and $\kappa_{\alpha}\rightarrow\kappa$ as $\alpha\rightarrow\infty$. Also, for large $\alpha$, $(r_{\alpha},s_{\alpha})\in(m,M)\times(m,M)$. For each $\alpha$ we now define the following smooth functions:
\begin{equation}\label{varphi}
\varphi_{\alpha}(r)=v_n(s_{\alpha})+\frac{\alpha}{2}(r-s_{\alpha})^2+\kappa_{\alpha}
\end{equation}
and
\begin{equation}\label{phi}
\phi_{\alpha}(s)=u(r_{\alpha})-\frac{\alpha}{2}(r_{\alpha}-s)^2-\kappa_{\alpha}.
\end{equation}
Then $\phi_{\alpha}\leq v_n$ and $\phi_{\alpha}(s_{\alpha})=v_n(s_{\alpha})$, and $u\leq\varphi_{\alpha}$ and $\varphi_{\alpha}(r_{\alpha})=u(r_{\alpha})$. This implies that
$$\min\left\{\delta u(r_{\alpha})-\mathcal{A}(\varphi_{\alpha},u)(r_{\alpha})-G(r_{\alpha}),u(r_{\alpha})-\mathcal{M}u(r_\alpha)\right\}\leq0$$
and
$$\min\left\{\delta v_n(s_{\alpha})-\mathcal{A}(\phi_{\alpha},v_n)(s_{\alpha})-G(s_{\alpha}),v_n(s_{\alpha})-\mathcal{M}v_n(s_{\alpha})\right\}\geq\frac{k}{n}.$$
Suppose that there exists $\alpha_0$ such that for all $\alpha\geq\alpha_0$ 
$$\delta u(r_{\alpha})-\mathcal{A}(\varphi_{\alpha},u)(r_{\alpha})-G(r_{\alpha})>0,$$
therefore we must have that $u(r_{\alpha})-\mathcal{M}u(r_\alpha)\leq0$. So, for any $\epsilon>0$ we can find $\alpha$ big enough and $0\leq\zeta_{\alpha}\leq M-r_{\alpha}$ such that
\begin{align*}
\kappa-\epsilon&\leq u(r_{\alpha})-v_n(s_{\alpha})\\
&\leq u(r_{\alpha}+\zeta_{\alpha})-C(r_{\alpha},\zeta_{\alpha})-v_n(s_{\alpha}+\zeta_{\alpha})+C(s_{\alpha},\zeta_{\alpha})-\frac{k}{n}\\
&\leq \kappa +\epsilon-\frac{k}{n},
\end{align*}
where the last inequality follows from the continuity of $C$. By choosing $\epsilon$ small enough this is a contradiction since $k>0$ and therefore such $\alpha_0$ does not exist. This leads to the existence of a subsequence where
$$\delta u(r_{\alpha})-\mathcal{A}(\varphi_{\alpha},u)(r_{\alpha})-G(r_{\alpha})\leq0.$$
Now, note that $\varphi'_{\alpha}(r_{\alpha})=\phi'_{\alpha}(s_{\alpha})=\alpha(r_{\alpha}-s_{\alpha})$. Then, we have that
\begin{align*}
(\delta+\lambda)(u(r_{\alpha})-v_n(s_{\alpha}))&\leq \alpha(c(s_{\alpha})-c(r_{\alpha}))(r_{\alpha}-s_{\alpha})+G(r_{\alpha})-G(s_{\alpha})\\
&+\lambda\left(\int_0^{r_{\alpha}-m}u(r_{\alpha}-x)dF(x)-\int_0^{s_{\alpha}-m}v_n(s_{\alpha}-x)dF(x)\right)\\
&\leq L\alpha(r_{\alpha}-s_{\alpha})^2+G(r_{\alpha})-G(s_{\alpha})\\
&+\lambda\int_0^{\infty}\left(1_{[r_{\alpha}-m]}u(r_{\alpha}-x)-1_{[s_{\alpha}-m]}v_n(s_{\alpha}-x)\right)dF(x).
\end{align*}
Taking $\alpha\rightarrow\infty$ along the subsequence, by Bounded Convergence Theorem we get
$$(\delta+\lambda)\kappa\leq\lambda\kappa,$$
and obtain a contradiction since $\kappa>0$. This concludes the proof.
\end{proof}

This proposition produces directly the following uniqueness theorem:

 \begin{theorem}\label{uniimp}
There is at most one viscosity solution $f$ of \eqref{hjbimp} for each given boundary conditions $f(m)$ and $f(M)$. 
\end{theorem}

\section{Optimal policy structure}\label{policy}

In this section we will describe the optimal maintenance policy structure by using the HJB equation \eqref{hjbimp}. We define the set $I=\{r\in[m,M]: V(r)=\mathcal{M}V(r)\}$, that is, the set of states where is optimal to intervene the system. Since $V$ and $\mathcal{M}V$ are continuous, the set $I$ is closed. We now define the following states:
\begin{equation}
\underline{s}=\min\{r:r\in I\}
\end{equation}
and
\begin{equation}
\overline{s}=\max\{r:r\in I\}.
\end{equation}
If $I$ is empty we define $\underline{s}=\overline{s}=M$. We also define the state $S$ as the smallest such that
\begin{align}\label{defS}
\nonumber
V(\overline{s})&=\max_{0\leq\zeta\leq M-\overline{s}}V(\overline{s}+\zeta)-H(\overline{s}+\zeta)+H(\overline{s})-k\\
&=V(S)-H(S)+H(\overline{s})-k\\
\nonumber
&=\mathcal{M}V(\overline{s}).
\end{align}
Note that $\overline{s}<S\leq M$ since $k>0$. Finally, we define the maintenance policy $\nu^*$ as follows: If the system is below $\underline{s}$ do nothing, if the system is in the interval $[\underline{s},\overline{s}]$ bring it to the state $S$ and if the system is above $\overline{s}$ do nothing. This novel policy structure has an interesting set of states, $[m,\underline{s})$ for which is not benefit-cost effective to do nay maintenance, that we call let-it-die states.

\begin{theorem} The $(\underline{s},\overline{s},S)$ policy $\nu^*$ is optimal.
\end{theorem}

\begin{proof}
We are going to show that $V(r)=J(r,\nu^*)$ for all $r\in[m,M]$. First of all, by definition of $I$ we have that if $r\in[m,\underline{s})$ then $J(r,\nu^*)=V(r)$ and by continuity of both functions we get $J(\underline{s},\nu^*)=V(\underline{s})$. Now, for $r\in[\underline{s},\overline{s}]$ we know that
$$J(r,\nu^*)=J(S,\nu^*)-H(S)+H(r)-k,$$
in particular for $\underline{s}$  we get
$$J(S,\nu^*)-k=J(\underline{s},\nu^*)+H(S)-H(\underline{s}).$$
Combining,
\begin{equation}\label{aux1}
J(r,\nu^*)=J(\underline{s},\nu^*)+H(r)-H(\underline{s})=V(\underline{s})+H(r)-H(\underline{s}).
\end{equation}
Since $J(\overline{s},\nu^*)\leq V(\overline{s})$, by \eqref{defS} and \eqref{aux1}
\begin{align*}
V(\underline{s})&\leq V(S)-H(S)+H(\underline{s})-k\\
&\leq \mathcal{M}V(\underline{s})=V(\underline{s}),
\end{align*}
hence 
\begin{equation}\label{equalf}
V(\underline{s})-H(\underline{s})=V(\overline{s})-H(\overline{s})=V(S)-H(S)-k
\end{equation}
and also 
\begin{equation}\label{equalS}
V(S)=J(S,\nu^*).
\end{equation} 
Suppose now that $r'\leq r$ with $r\in I$, then
\begin{align*}
V(r')-H(r')&\geq \max_{0\leq\zeta\leq M-r'}V(r'+\zeta)-H(r'+\zeta)-k\\
&=\max_{r'\leq\zeta\leq M}V(\zeta)-H(\zeta)-k\\
&\geq \max_{r\leq\zeta\leq M}V(\zeta)-H(\zeta)-k\\
&=V(r)-H(r).
\end{align*}
Therefore, for $r\in I$ such that $\underline{s}<r<\overline{s}$ from \eqref{equalf} we get that $V(r)-H(r)=V(S)-H(S)-k$ and by \eqref{equalS}
\begin{equation}
V(r)=J(S,\nu^*)-H(S)+H(r)-k=J(r,\nu^*).
\end{equation}
It only remains to show that $I=[\underline{s},\overline{s}]$. If $r\in[\underline{s},\overline{s}]\setminus I$ then $V(r)>J(r,\nu^*)$, so let $\tau^*=\inf\{t\geq0:R^S_t\leq\overline{s}\}$ and by Lemma \ref{lema}
\begin{align*}
V(S)&=\mathbb{E}_S\left[\int_0^{\tau^*}e^{-\delta s}G(R^S_s)ds+e^{-\delta\tau^*}V(R^S_{\tau})\right]\\
&=\mathbb{E}_S\left[\int_0^{\tau^*}e^{-\delta s}G(R^S_s)ds+e^{-\delta\tau^*}\left(1_{\{R^S_{\tau^*}\in I\cup[m,\underline{s})\}}J(R^S_{\tau^*},\nu^*)+1_{\{R^S_{\tau^*}\in[\underline{s},\overline{s}]\setminus I\}}V(R^S_{\tau^*})\right)\right]\\
&>\mathbb{E}_S\left[\int_0^{\tau^*}e^{-\delta s}G(R^S_s)ds+e^{-\delta\tau^*}J(R^S_{\tau},\nu^*)\right]\\
&=J(S,\nu^*)=V(S).
\end{align*}
Since the support of $F$ contains the interval $[0,M-m]$, this contradiction implies that the open set $[\underline{s},\overline{s}]\setminus I=\emptyset$ and this concludes the proof.
\end{proof}

\section{Numerical examples}\label{numer}

In this section we present some numerical examples to illustrate the results of the paper and perform a sensitivity analysis of the fixed cost $k$. Let us start by denoting $V_k$ and $C_k$ the value function and the cost function with fixed cost $k$, so we have the following convergence result.

\begin{proposition}\label{limit}
$V_k$ converges point-wise to $V_0$ as $k$ approaches to 0.
\end{proposition}
\begin{proof}
First, it is clear that $V_k$ is an increasing sequence of functions as $k\downarrow0$, with upper bound $V_0$. Let $\Upsilon_m$ the set of admissible strategies with at most $m$ interventions, hence by Lemma 7.1 in \cite{oeksendal05acjd}, for any $\epsilon>0$, there exist $m\geq0$ and $\nu\in\Upsilon_m$ such that
\begin{align*}
V_0(r)&\leq\mathbb{E}_{r}\left[\int_0^{\tau^{\nu}}e^{-\delta s}G(R_s^{\nu})ds-\sum\limits_{\tau_i<\tau^{\nu}}e^{-\delta\tau_i}C_0(R^{\nu}_{\tau_i},\zeta_i)\right]+\epsilon\\
&=\mathbb{E}_{r}\left[\int_0^{\tau^{\nu}}e^{-\delta s}G(R_s^{\nu})ds-\sum\limits_{\tau_i<\tau^{\nu}}e^{-\delta\tau_i}C_k(R^{\nu}_{\tau_i},\zeta_i)\right]+k\mathbb{E}\left[\sum\limits_{\tau_i<\tau^{\nu}}e^{-\delta\tau_i}\right]+\epsilon\\
&\leq V_k(r)+km+\epsilon.
\end{align*}
Then, for any $k\leq\frac{\epsilon}{m}$ we have the result since $\epsilon$ is arbitrary.
\end{proof}

To approximate the value function $V_k$ we follow the Jacobi value function-iteration method described in \cite{Kushner} (see also \cite{Junca} for more details) until an error less that $10^{-7}$ is reached. We will present three examples with the same data except for the cost function $H$. The system performance is measured as a percentage of the optimal level, so $m=0$ and $M=1$. Shocks are assumed to arrive with intensity $\lambda=0.5$ and size distributed Lognormal(-1,0.14). The progressive deterioration is described by the function $c(r)=0.05(1+\epsilon-r)$ with $\epsilon=10^{-6}$ (remember that $c>0$ so we need to add this $\epsilon$). The benefit function is defined as $G(r)=C(1-e^{\frac{r}{2}})$ with $C$ varying for each example. Finally, the discount factor is $\delta=0.05$.

\begin{figure}[t]
   \begin{subfigure}[b]{.5\linewidth}
      \includegraphics[width=1.1\textwidth]{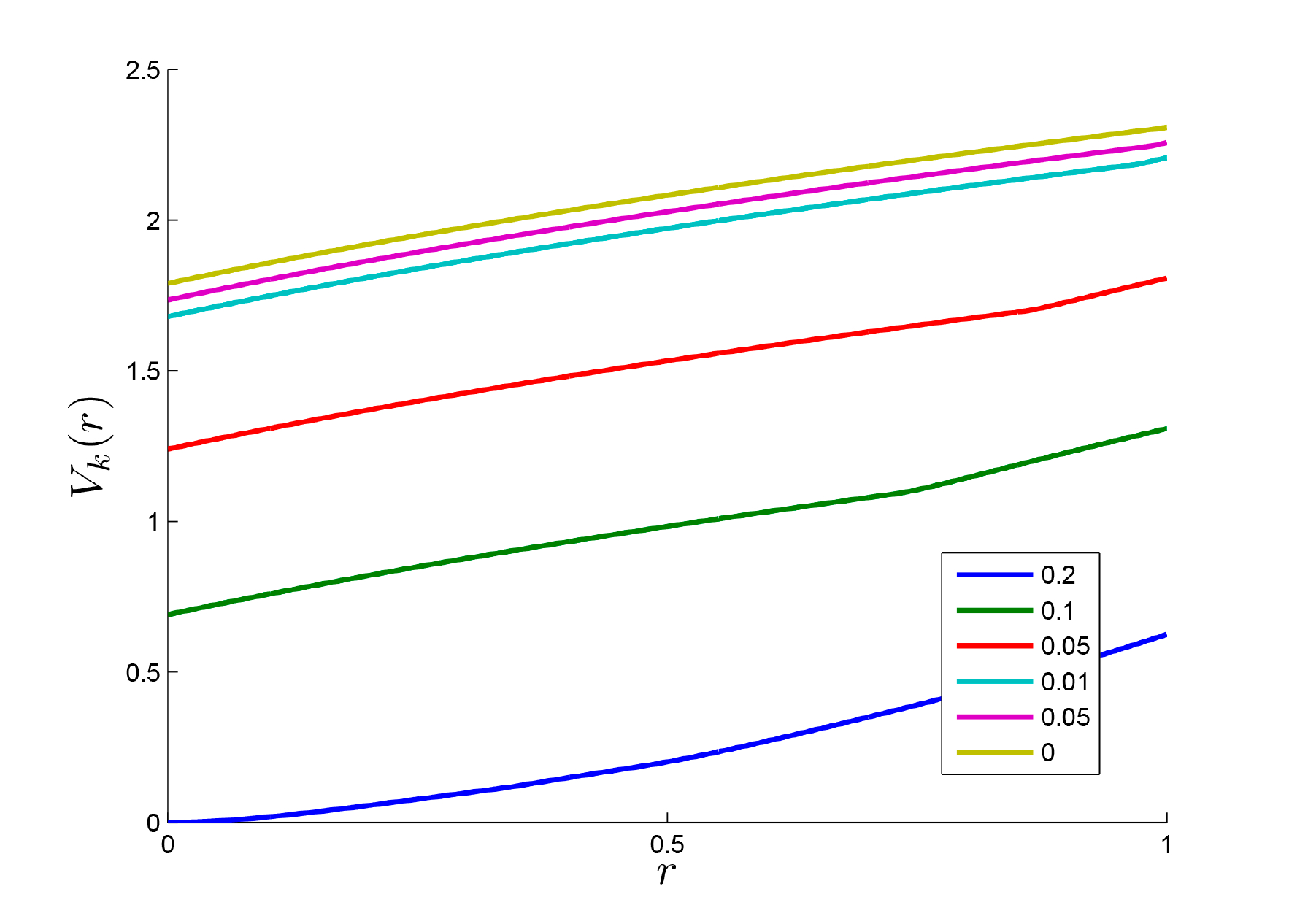}
   \end{subfigure}
\begin{subfigure}[b]{.5\linewidth}
      \includegraphics[width=1.1\textwidth]{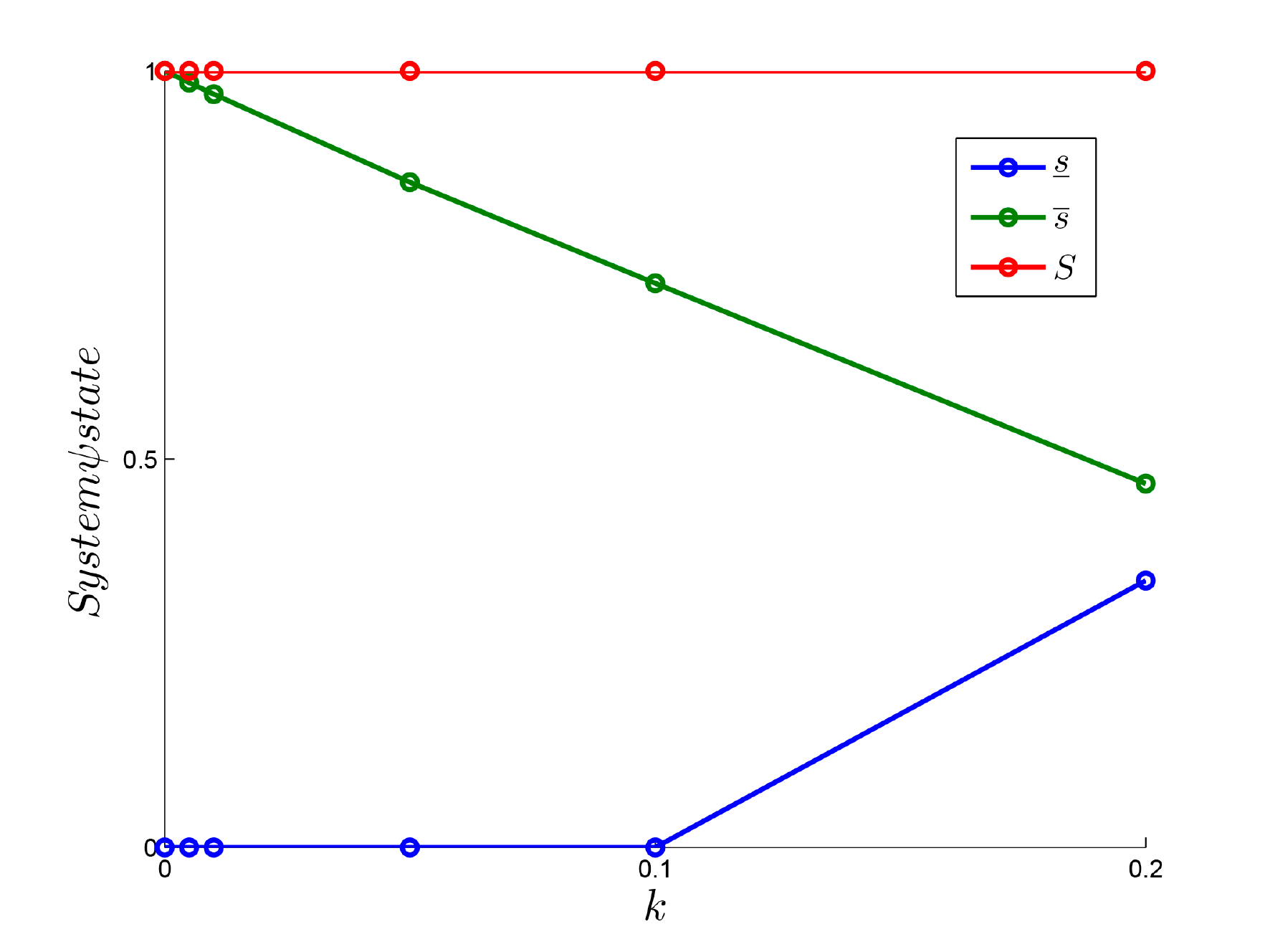}
   \end{subfigure}
      \caption{Value function and $(\underline{s},\overline{s},S)$ optimal policy with $H(r)=\sqrt{r+0.5}$ and $C=0.5$}\label{figraiz}
\end{figure}

Figure \ref{figraiz} shows the value function $V_k$ for different values of $k$ and maintenance cost $H(r)=\sqrt{r+0.5}$. It also shows the values of the optimal $(\underline{s}_k,\overline{s}_k,S_k)$ policy for each value of $k$. The first thing we can notice is the linear behavior of $\overline{s}_k$. We also see that $S_k=1$ for all $k$, so is optimal to bring the system to perfect condition at each maintenance. Now, note that as $k$ goes to 0, the optimal policy becomes an always-repair policy since $\underline{s}_k$ goes to 0 and $\overline{s}_k$ goes to 1. This implies that for all $r\in[0,1]$ we have that $V_0(r)=V_0(1)-H(1)+H(r)$, that is, $V_0$ is just the cost function plus a constant term.

\begin{figure}[h]
   \begin{subfigure}[b]{.5\linewidth}
      \includegraphics[width=1.1\textwidth]{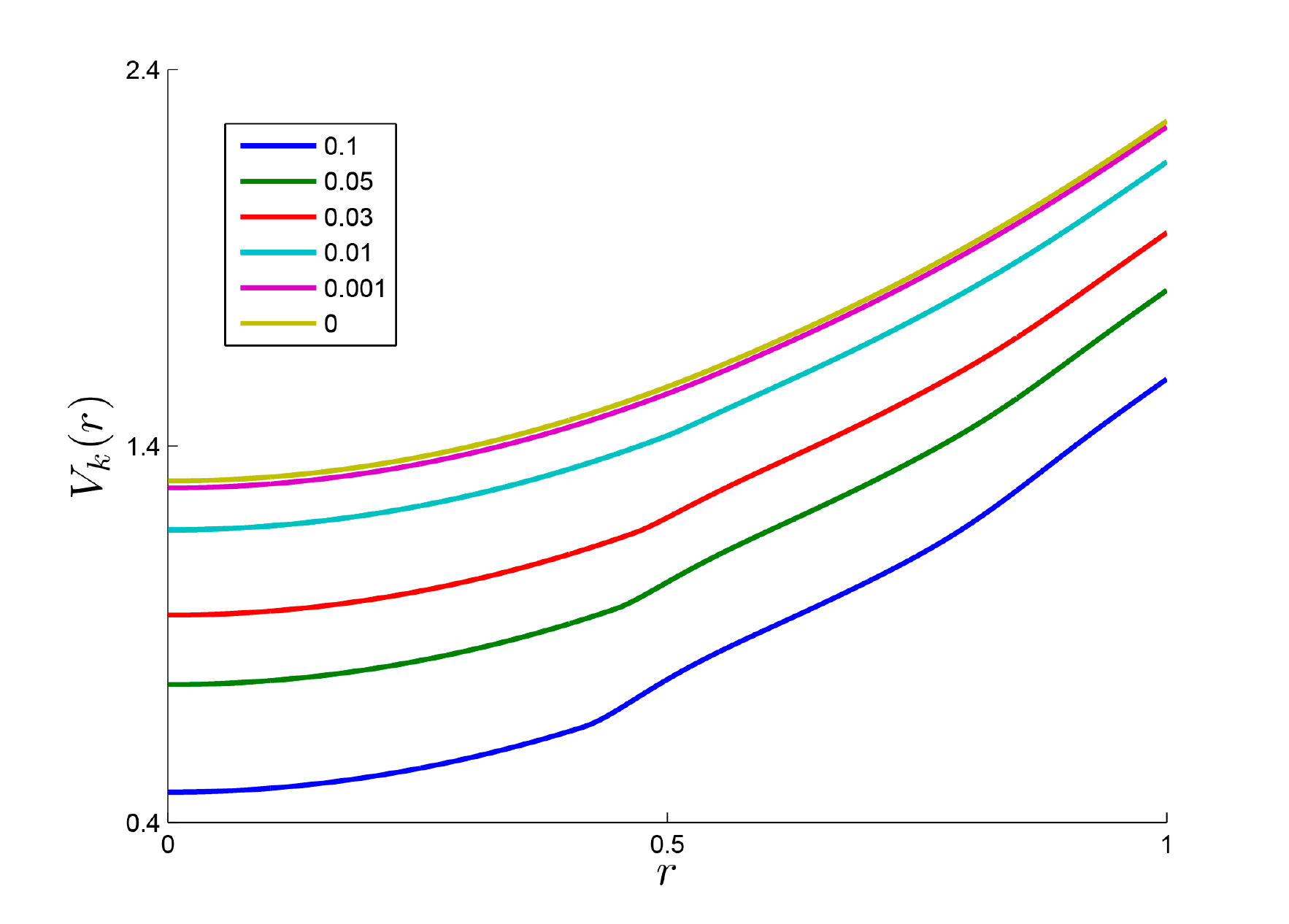}
   \end{subfigure}
\begin{subfigure}[b]{.5\linewidth}
      \includegraphics[width=1.1\textwidth]{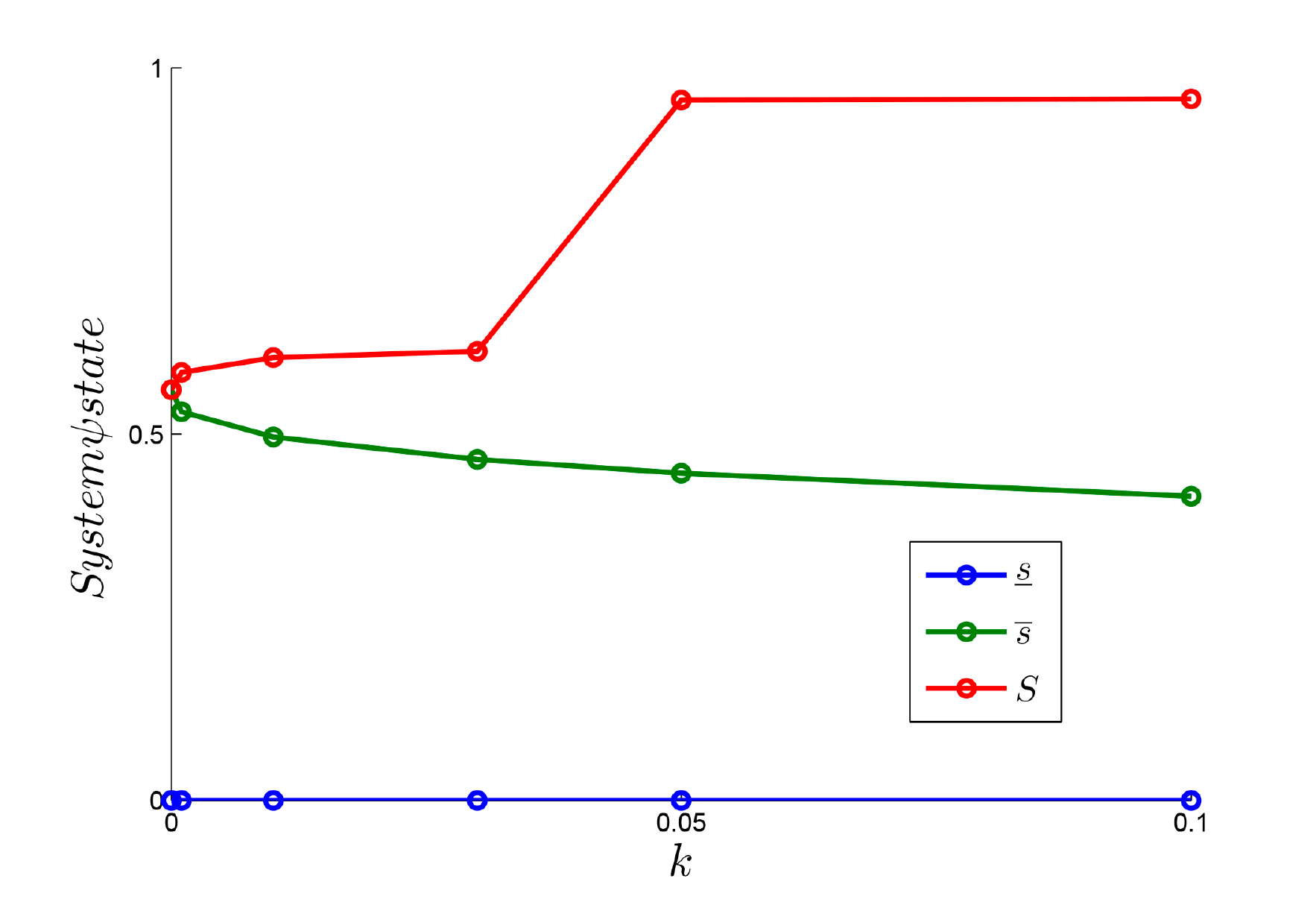}
   \end{subfigure}
      \caption{Value function and $(\underline{s},\overline{s},S)$ optimal policy with $H(r)=r^2$ and $C=1$}\label{figcuad}
\end{figure}

Now, when $H(r)=r^2$ and $C=1$, the results are as shown in Figure \ref{figcuad}. In this case $H$ is a convex function increasing rapidly for $r$ close to 1, so for big $k$ is optimal to bring the system to optimal conditions, but as $k$ decreases we can afford more frequent repairs so $S_k$ also decreases. The interesting fact is that it seems to be a threshold for which $S_k$ jumps from 1 to a much smaller value.

\begin{figure}[h]
   \begin{subfigure}[b]{.5\linewidth}
      \includegraphics[width=1.1\textwidth]{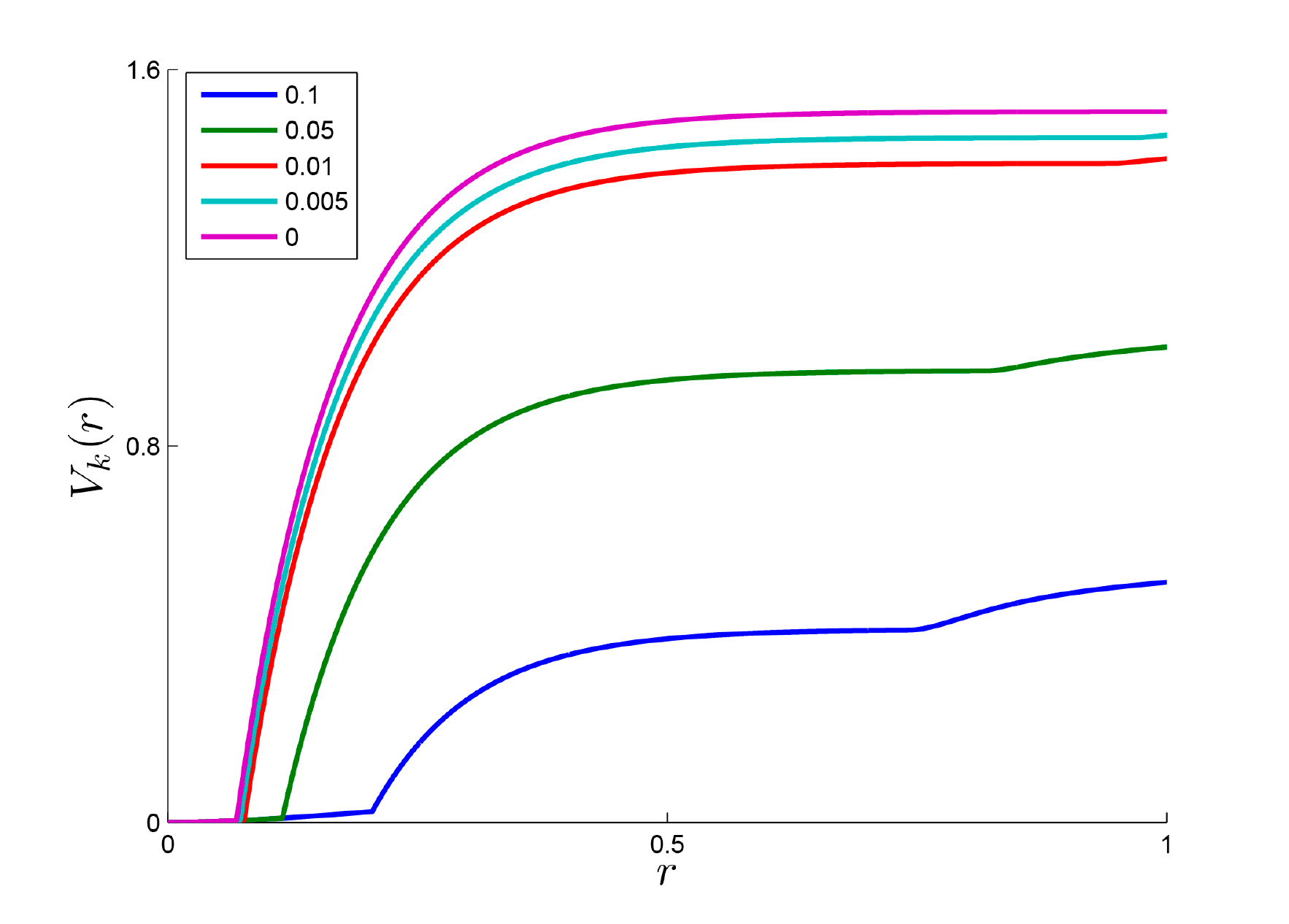}
   \end{subfigure}
\begin{subfigure}[b]{.5\linewidth}
      \includegraphics[width=1.1\textwidth]{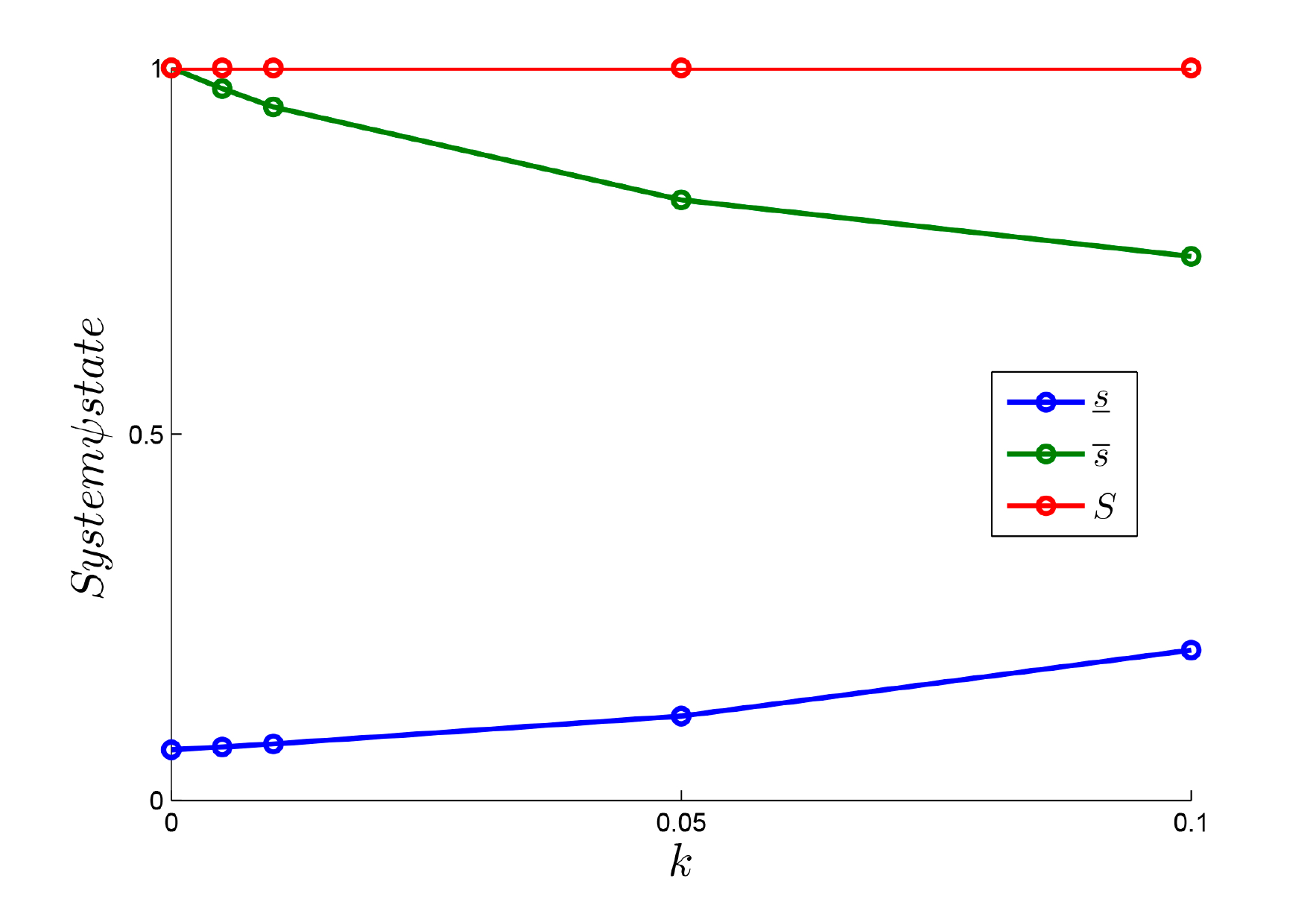}
   \end{subfigure}
      \caption{Value function and $(\underline{s},\overline{s},S)$ optimal policy with $H(r)=3(1-e^{-10r})$ and $C=0.2$}\label{figexp}
\end{figure}

Finally, Figure \ref{figexp} shows that when $H(r)=3(1-e^{-10r})$ and $C=0.2$ we have that $\underline{s}_k$ is bigger than zero, so there are some valuer of $r$ for which is optimal to do nothing and let the system failure. This is explained because of the very rapidly increasing values of $H$ for small values of $r$. We also see in this case that $V_k$ is clearly not in $C^1$ for any $k$ as opposed to the previous cases where the values function becomes smoother when $k$ decreases.

\section{Conclusions and future work}\label{future}

In this work we extend some previous works of \cite{Junca,Junca20131} to include progressive deterioration in a system subject to shocks that is maintained using impulse control strategies with fixed cost. We use the theory of viscosity solution to characterize the value function as the unique solution of the HJB equation associated with the control problem. The main result is the proof that optimal strategies in this model always has an $(\underline{s},\overline{s},S)$ structure with a set of let-it-die states.

Several questions remain still open. In the case of no fixed cost, by Remark \ref{k0} we know that $V_0$ is a viscosity solution of the HJB equation but we will not able to characterize it with this equation. In this case we have that $$\mathcal{M}_0f(r):=\sup_{0\leq\zeta\leq M-r}f(r+\zeta)-H(r+\zeta)+H(r)\geq f(r),$$
for any measurable function $f$. On the other hand, we must have that $V_0\geq\mathcal{M}_0V_0$. Therefore, $V_0=\mathcal{M}_0V_0$ and this shows that in fact the HJB equation \eqref{hjbimp} does not provide enough information to characterize $V_0$. Also, $\zeta=0$ attains the maximum value in the computation of $\mathcal{M}_0V_0(r)$ for all $r\in[m,M]$, so it is always optimal to intervene the system but at the same time do nothing is optimal, which is a contradiction. Another consequence is that we are not able to use this condition in order to find an optimal policy for this case, so further information is needed.

Following the same idea as in \cite{Junca11}, assume that $H$ and $V_0$ are differentiable functions. Then for all $r\in(m,M)$
\begin{equation}\label{regcond}
0\geq\left.V_0'(r+\zeta)-H'(r+\zeta)\right|_{\zeta=0}=V_0'(r)-H'(r).
\end{equation}

Hence, we get that $V_0$ must satisfy the condition $h-V_0'\geq0$, where we define $h=H'$. Since $H$ is increasing then $h$ is a positive function. This suggests that if we assume no fixed cost we should look at an equation different of \eqref{hjbimp} and consider the HJB equation
\begin{equation}\label{hjbsing}
\min\left\{\delta f(r)-\mathcal{A}f(r)-G(r),h(r)-f'(r)\right\}=0.
\end{equation}
In fact, the equation \eqref{hjbsing} is the associated equation of the following singular control problem (see \cite{Taksar}). In this case the \emph{admissible} controls $\xi_t$ are predictable, non-decreasing and c\`agl\`ad processes with $\xi_{0}=0$, such that the process
$$R_t^{\xi}=r-\int_0^tc(R_s^{\xi})ds-\sum\limits_{i=1}^{N_t}S_i+\xi_t$$
stays in the interval $[m,M]$. We define the total failure time $\tau^{\xi}$ analogously as above. Now, the form of the value function $V_s$ changes to
\begin{align}\label{valsing}
V_s(r)&=\sup_{\xi}J(r,\xi)\\\nonumber
&=\sup_{\xi}\mathbb{E}_r\left[\int_{0}^{\tau^{\xi}}e^{-\delta s}G(R_s^{\xi})ds-\int_{0}^{\tau^{\xi}}e^{-\delta s}h(R_{s}^{\xi})d\xi^c_s-\sum\limits_{s<\tau^{\xi}\atop \xi_{s+}\neq\xi_s}e^{-\delta s}(H(R_{s+}^{\xi})-H(R_{s}^{\xi}))\right]
\end{align}
for all $r\in[m,M]$, where $\xi^c$ is the continuous part of the control process. Note that the last term in \eqref{valsing} is the same as in the impulse control problem when $k=0$. As before, $V_s$ is a bounded and non-negative function. Finally,  it is clear that $V_0\leq V_s$ since any admissible impulse control is also an admissible singular control. So, the first question is if both are in fact the same. For instance, the first numerical example in the previous section shows that $V_0$ is just the cost function plus a constant term, therefore $V_0'-h=0$. Since $V_0$ satisfies \eqref{hjbimp} we conclude that $V_0$ also satisfies \eqref{hjbsing}.

Now, the figures above show that the sequences $\{\underline{s}_k\}$ and $\{S_k\}$ are non-increasing while the sequence $\{\overline{s}_k\}$ is non-decreasing. Is this true in general? If so, they will be convergent with $\underline{s}_k\rightarrow \underline{s}$ and $\overline{s}_k\rightarrow\overline{s}\leftarrow S_k$. Is $(\underline{s},\overline{s})$ an optimal policy for the problem with no fixed cost? Can this policy be extracted from the HJB equation \eqref{hjbsing}? If the answer is yes, then we will also have an atypical set of let-it-die states given by $[m,\underline{s})$. This situation occurs for example in the optimal dividends payment problem for insurance companies in some cases, see \cite{Azcue}.

Another direction for future research is to consider a model that is not permanently observed but only at discrete times (which can also be part of the decision).

\appendix

\section{Proof of Lemma \ref{lema}}\label{applema}
\begin{proof}
Given $\tau$ a stopping time with respect to the filtration $\mathcal{F}_t$, let $\mathcal{I}_{\tau}$ the set of admissible controls with $\tau_1\geq\tau$ a.s.

Let $r\in[m,M]$ and $\nu\in\mathcal{I}_{\tau}$, then 
\begin{equation}\label{lemadyn}
V(r)\geq J(r,\nu)=\mathbb{E}_r\left[\int_0^{\tau\wedge\tau^0}e^{-\delta s}G(R^r_s)ds+1_{\{\tau<\tau^0\}}e^{-\delta\tau}\eta^{\nu} \right],
\end{equation}
where
$$\eta^{\nu}=\int_0^{\tau^{\nu}-\tau}e^{-\delta s}G(R_s^{\nu})ds-\sum\limits_{\tau_i<\tau^{\nu}}e^{-\delta(\tau_i-\tau)}C(R^{\nu}_{\tau_i},\zeta_i).$$

From the strong Markov property of the process $R_s$ (see \cite{rogerswilliams}) we have that
\begin{align*}
\mathbb{E}_r\left[1_{\{\tau<\tau^0\}}e^{-\delta\tau}\eta^{\nu} \right]&=\mathbb{E}_r\left[1_{\{\tau<\tau^0\}}e^{-\delta\tau}\mathbb{E}[\eta^{\nu}|\mathcal{F}_{\tau}] \right]\\
&=\mathbb{E}_r\left[1_{\{\tau<\tau^0\}}e^{-\delta\tau}\mathbb{E}_{R^r_{\tau}}[\eta^{\nu}] \right]\\
&=\mathbb{E}_r\left[1_{\{\tau<\tau^0\}}e^{-\delta\tau}J(R^r_{\tau},\bar{\nu})\right],
\end{align*}
where $\bar{\nu}$ has the intervention times of $\nu$ subtracted by $\tau$. Hence, for an $\epsilon-$optimal strategy with initial level $R^r_{\tau}$ we have that
$$V(r)\geq\mathbb{E}_r\left[\int_0^{\tau\wedge\tau^0}e^{-\delta s}G(R^r_s)ds+1_{\{\tau<\tau^0\}}e^{-\delta\tau}(V(R^r_{\tau})-\epsilon)\right].$$
Since $\epsilon$ is arbitrary we obtain \eqref{dynamic}. If it is optimal not to intervene before $\tau$, then any $\epsilon-$optimal strategy belongs to $\mathcal{I}_{\tau}$ and we obtain the reverse inequality.
\end{proof}

\bibliography{referencias}
\bibliographystyle{abbrv}

\end{document}